\documentclass[12pt]{amsart}
\usepackage{amsthm,amssymb,verbatim,amsmath, amsfonts, amscd,color}

\usepackage{amssymb}
\usepackage{amsthm}

\usepackage{hyperref}
\usepackage{verbatim}
\usepackage{comment}
\theoremstyle{plain}\newtheorem{Theorem}{Theorem}[section]
\theoremstyle{plain}
\theoremstyle{plain}\newtheorem{Corollary}[Theorem]{Corollary}
\theoremstyle{plain}\newtheorem{Lemma}[Theorem]{Lemma}
\theoremstyle{plain}\newtheorem{Proposition}[Theorem]{Proposition}
\theoremstyle{definition}
\theoremstyle{definition}\newtheorem{Example}[Theorem]{Example}
\theoremstyle{definition}
\theoremstyle{definition}\newtheorem{Remark}[Theorem]{Remark}
\theoremstyle{definition}
\theoremstyle{definition}\newtheorem{Notation}[Theorem]{Notation}
\theoremstyle{definition}
\theoremstyle{definition}
\theoremstyle{definition}
\theoremstyle{definition}
\theoremstyle{definition}
\theoremstyle{definition}
\theoremstyle{definition}\newtheorem{Notation/Definition}
[Theorem]{Notation/Definition}
\theoremstyle{definition}

\def\CF{{\mathcal{F}}}    
\def\FM{{\mathfrak{M}}}

\def\ker{\mathrm{ker}}

\def\Ind{\mathrm{Ind}}

\def\Res{\mathrm{Res}}
\def\Inf{\mathrm{Inf}}

\def\tr{\mathrm{tr}}

\newcommand{\Sc}{{\text{Sc}}}

\begin{document}
\title{Recognition of Brauer indecomposability for a Scott module }
\date{\today}
\author{{Shigeo Koshitani and {\.I}pek Tuvay}}
\address{Department of Mathematics and Informatics, Graduate School of Science,       
Chiba University, 1-33 Yayoi-cho, Inage-ku, Chiba 263-8522, Japan.}
\email{koshitan@math.s.chiba-u.ac.jp}
\address{Mimar Sinan Fine Arts University, Department of Mathematics, 34380, Bomonti, \c{S}i\c{s}li, Istanbul, Turkey}
\email{ipek.tuvay@msgsu.edu.tr}

\thanks{The first author
was partially supported by the Japan Society for Promotion of Science (JSPS),
Grant-in-Aid for Scientific Research (C)19K03416, 2019--2021. }
\dedicatory{In the memory of Professor Hiroyuki Tachikawa}

\keywords{Brauer indecomposability, Scott modules}
\subjclass[2010]{20C20, 20C05}

\begin{abstract}
We give a handy way to have a situation that the
$kG$-Scott module with vertex $P$ remains
indecomposable under taking the Brauer construction for any subgroup $Q$ of $P$
as $k[Q\,C_G(Q)]$-module,
where $k$ is a field of characteristic $p>0$.
The motivation
is that the Brauer indecomposability of
a $p$-permutation bimodule is one of the key steps in order
to obtain a splendid stable equivalence of Morita type 
by making use of the gluing method,
that then can possibly lift to a splendid derived equivalence.
Further our result explains a hidden reason why the Brauer indecomposability of
the Scott module fails in Ishioka's recent examples.
\end{abstract}

\maketitle

\section{Introduction and notation}

\noindent
In modular representation theory of finite groups,
the Brauer construction $M(P)$ of a $p$-permutation $kG$-module $M$ with respect to 
a $p$-subgroup $P$ of a finite group $G$ plays a very important role,
where $k$ is an algebraically closed field of characteristic $p>0$.
It canonically becomes a $p$-permutation module over $kN_G(P)$ (see \cite[p.402]{Bro}).

In their paper \cite{KKM} Kessar, Kunugi and Mitsuhashi introduce a notion {Brauer indecomposability}.
Namely, $M$ is called {\it Brauer indecomposable} if the restriction
module ${\mathrm{Res}}\,^{N_G(Q)}_{Q\,C_G(Q)}\,M(Q)$ is indecomposable or zero
as $k(Q\,C_G(Q))$-module for any subgroup $Q$ of $P$.
Actually in order to get a kind of equivalence between
the module categories for the principal blocks $A$ and $B$ 
of the group algebras $kG$ and $kH$, respectively
(where $H$ is another finite group), e.g. 
in order to prove Brou\'e's abelian defect group conjecture, 
we usually first of all  would
have to face a situation such that $A$ and $B$
are stably equivalent of Morita type.
In order to do it 
we often want to check whether the $k(G\times H)$-Scott module
${\mathrm{Sc}}(G\times H, \Delta P)$ with a vertex  
$\Delta P:=\{(u,u)\in P\times P \}$ induces a stable
equivalence of Morita type between $A$ and $B$, where
$P$ is a common Sylow $p$-subgroup of $G$ and $H$
by making use of 
the gluing method due to Brou\'e (see \cite[6.3.Theorem]{Bro94}), and also
Rickard, Linckelmann and Rouquier \cite[Theorem 5.6]{Rou}.
If this is the case, then ${\mathrm{Sc}}(G\times H, \Delta P)$
has to be Brauer indecomposable
(as done in \cite{KLi, KKM, KKL, KL1, KL2, KLS}).
Therefore it should be very important to know if
${\mathrm{Sc}}(G\times H, \Delta P)$
is Brauer indecomposable or not.
Even more is that one can understand what happens in interesting Ishioka's examples \cite{I} 
from main result's point of view.
These motivate us to write this paper.

Actually, our main result is the following:

\begin{Theorem}\label{MainResult}
Let $G$ be a finite group and $M:={\mathrm{Sc}}(G,P)$ the
Scott $kG$-module with vertex $P$. 
Then the following 
are equivalent:
\begin{enumerate}
\item $M$ is Brauer indecomposable.
\item There exists an $R\unlhd G$ with $R\leq P\cap\ker(M)$ satisfying that
for every $Q$ with $R\leq Q\leq P$, $\Res^{N_G(Q)}_{C_G(Q)}\, M(Q)$ is indecomposable.
\end{enumerate}
\end{Theorem}

\begin{Corollary}\label{indexP} 
Let $G$ be a finite group and $M:={\mathrm{Sc}}(G,P)$ the
Scott $kG$-module with vertex $P$.
 Suppose that $p\,{\not|}\,|N_G(P)/P\,C_G(P|$ and that
there exists an $R\unlhd G$ with $R\leq P\cap\ker(M)$ and $|P/R|=p.$
Then the following are equivalent:
\begin{enumerate}
\item $M$ is Brauer indecomposable.
\item $\Res^{G}_{C_G(R)}\, M$ is indecomposable,
\end{enumerate}
 
\end{Corollary}

\noindent
These generalize \cite{I} and closely related to \cite{IK}, and 
there are results on Brauer indecomposability of Scott modules also
in \cite{KT19a, KT19b, KT21, T}.

\begin{Notation} Besides the notation explained above we need the following notation and terminology.
In this paper $G$ is always a finite group, $k$ is an algebraically closed field of characteristic $p>0$
and $kG$ is the group algebra of $G$ over $k$.

By $H\leq G$ and $H \unlhd G$ we mean that $H$ is a subgroup and a normal subgroup of $G$, 
respectively. We mean by $H<G$ that $H$ is a {\it proper} subgroup of $G$.
We write $O_{p}(G)$ for the largest normal ${p}$-subgroup of $G$.
For $x, y\in G$ we set $^y\!x:= yxy^{-1}$ and $x^y:=y^{-1}xy$ .
Further for $H\leq G$ and $g\in G$ 
we set $^g\!H:=\{\, ^g\!h \, |\, \forall h\in H \}$.
For two groups $K$ and $L$ we write $K\rtimes L$ for a semi-direct product of $K$ by $L$
where $K \unlhd (K\rtimes L)$.
For an integer $n\geq 1$ we mean by $C_n$, $S_n$ and $A_n$
the cyclic group of order $n$, the symmetric and alternating groups 
of degree $n$, respectively.
Further  $D_{2^n}$ is the dihedral group of order $2^n$.

We mean by a $kG$-module a finitely generated left $kG$-module. 
For a $kG$-module $M$ and a $p$-subgroup $P$ of $G$ the Brauer construction $M(P)$ is
defined as in \cite[\S27]{Th} or in \cite[p.402]{Bro}. For such $G$ and $P$ we write 
$\mathcal F_P(G)$ for
the fusion system of $G$ over $P$ as in \cite[I.Definition 1.1]{AKO}.
For two $kG$-modules $M$ and $L$, we write $L\,|\,M$ if $L$ is (isomorphic to) 
a direct summand of $M$ as a $kG$-module.
We write $k_G$ for the trivial $kG$-module, and $1_G$ for the ordinary trivial character of $G$. 
For $H\leq G$, a $kG$-module $M$ and a $kH$-module $N$, 
we write ${\mathrm{Res}}^G_H\,M$ for the restriction module of
$M$ from $G$ to $H$, and ${\mathrm{Ind}}_H^G\,N$ 
for the induced module of $N$ from $H$ to $G$.
When $K\leq H\leq G$ and $M$ is a $kG$-module, 
we can define the (relative) trace map ${\mathrm{tr}}_K^H$
for an element $m\in M^K$, that is, 
${\mathrm{tr}}_K^H\,(m):=\sum_{h\in [H/K]}\,hm$ where
$[H/K]$ is the coset representatives and 
$M^K:=\{m\in M\,|\,\kappa m=m \ \forall\kappa\in K\}$
(see \cite[\S 18]{Th}). 
When $M$ is a $kG$-module, 
the kernel of $M$ is defined by 
$\ker(M):=\{g\in G\,|\, gm=m \ \forall m\in M\}$.

For $H\leq G$ we denote by ${\mathrm{Sc}}(G,H)$ 
the (Alperin-)Scott $kG$-module with respect to $H$.
By definition, ${\mathrm{Sc}}(G,H)$ is the unique indecomposable direct summand
of ${\mathrm{Ind}}_H^G\,k_H$ which contains $k_G$ in its top.
We refer the reader to 
\cite[\S2] {Bro} and \cite[Chap.4, \S 8.4]{NT} for further details on Scott modules.

For the other notations and terminologies, see the books
\cite{NT} 
and \cite{AKO}.
\end{Notation}
The organization of this paper is as follows.
In \S2 we shall give several lemmas
that are useful of our aim, and 
we shall prove our main theorem and a couple of examples in \S3.

\section{Lemmas}
Throughout this section we fix a finite group $G$ and its $p$-subgroup $P$, and further set
$M:=\Sc(G,P)$ and $\mathcal F:=\mathcal F_P(G)$.

\begin{Lemma}\label{kernelP}
If $p\,{\not |}\,\,|N_G(P)/P \, C_G(P)|$ and $P\leq\ker(M)$,
then $M$ is Brauer indecomposable. 
\end{Lemma}
\begin{proof} 
It holds $\Res^G_{N_G(P)}\, M = M(P)$
by \cite[Exercise (27.4)]{Th}. So that \cite[Lemma 4.3(i)]{KKM} yields that
\begin{equation}\label{C_G(P)}
\Res^G_{C_G(P)}\,M \text{ is indecomposable.}
\end{equation}
Now, take any $Q\leq P$. Clearly $Q\leq\ker (M)$, so that 
$\Res^G_{N_G(Q)}\,M = M(Q)$
as above. Thus, 
$\Res^{N_G(Q)}_{C_G(Q)} M(Q) = \Res^G_{C_G(Q)}\,M$,
which implies from (\ref{C_G(P)}) that 
$\Res^G_{C_G(Q)}\,M$ is indecomposable
since $C_G(P)\leq C_G(Q)$. Namely $\Res^{N_G(Q)}_{C_G(Q)}\,M(Q)$ is
indecomposable.
\end{proof}

\begin{Corollary}\label{p'group}
If $P\unlhd G$ and $p\,{\not|}\,\,|G/P \, C_G(P)|$,
then $M$ is Brauer indecomposable. 
\end{Corollary}
\begin{proof}
Since $P\unlhd G$ we have that 
$M \ | \ \Ind_P^G\,k \cong k(G/P)$ so that $P\leq \ker(M)$ and 
the result follows from Lemma \ref{kernelP}.
\end{proof}

\begin{Example}\label{S4}
Suppose that $G:=S_4$, $p=2$ and $P:=O_2(G)$
(then $P\cong C_2\times C_2$ and
$G=P\rtimes S_3$, so that $A_4=P\rtimes C_3<G$).
Set $H:=A_4$, then $M=\Sc(G,H)$ by \cite[Chap.4, Corollary 8.5]{NT}.
Note that $P\vartriangleleft G$ and $M(P)=M$ by the definition of Brauer construction (quotient).
So that Green's indecomposability theorem implies that
$M=\Ind_H^G\,k_H \cong \boxed{\begin{matrix} k_G\\k_G\end{matrix}}$ 
(a uniserial $kG$-module of length $2$
with two composition factors $k_G$) that affords the ordinary character
$1_G+\chi_1$ where $\chi_1$ is (so-called) the sign character of degree $1$.
Since $\Res^G_P\,1_G = \Res^G_P\,\chi_1 = 1_P$, it follows from
\cite[II Theorem 12.4(iii)]{Lan} that 
$\Res^G_P\,M\cong k_P\oplus k_P.$

Thus $\Res^{N_G(P)}_{C_G(P)}M(P)=\Res^G_P\,M\cong k_P\oplus k_P$,
namely, $M$ is {\it not} Brauer indecomposable.
Note that although $P$ is normal in $G$, we have that $N_G(P)/P\,C_G(P)=G/P\cong S_3$ which 
is not a $2'$-group. See Corollary \ref{p'group}.
\end{Example}

\begin{Corollary}\label{saturatedness}
If $P\unlhd G$, then the following are equivalent.
\begin{enumerate}
\item $\mathcal F:= \mathcal F_P(G)$ is saturated.
\item $p\,{\not|}\,|G/P\,C_G(P)|$.
\item $M$ is Brauer indecomposable.
\end{enumerate}
\end{Corollary}

\begin{proof}
$(1)\Rightarrow (2):$ Follows from the Sylow Axiom for saturated fusion systems.

$(2) \Rightarrow (3):$ Follows from Corollary \ref{p'group}.

$(3) \Rightarrow (1):$ Follows from \cite[Theorem 1.1]{KKM}.
\end{proof}

\begin{Remark}
If $P$ is a normal $p$-subgroup of $G$, then any $\mathcal F$-isomorphism 
can be extended to an 
$\CF$-automorphism of $P$, so in particular $\CF$ satisfies the 
Extension Axiom.
As a result, the equivalence of $(1)$ and $(2)$ of Corollary \ref{saturatedness} 
is obtained automatically (see \cite[\S 2 of Part I]{AKO}).
\end{Remark}

\begin{Lemma}\label{QR}\label{tr1}
Let $\FM$ be an indecomposable $p$-permutation $kG$-module with vertex $P$, and let $R\unlhd G$ 
be a $p$-subgroup such that $R\leq \ker(\FM)$. Assume that $Q\leq P$.
Then for any $H$ such that $R \not \leq H < QR$, we have that $\tr_H^{QR}(\FM^H)=0$.
\end{Lemma}

\begin{proof}
Since $R\not\leq H$, we have $H < HR$. The representatives of the set of left cosets 
of $H$ in $HR$ can be chosen from $R$ and the number of the left cosets is divisible by $p$. 
So, since $R$ acts trivially on $\FM$
$$\tr_H^{HR}(m)=\sum_{r} r m=\sum_{r} m=0 \text{ for every }m\in \FM^H$$
where $r$ runs over the different left cosets of $H$ in $HR$. Then 
$$\tr_H^{QR}(\FM^H)= (\tr_{HR}^{QR}\circ\tr_H^{HR})(\FM^H) =0$$ as desired.
\end{proof}

\begin{Lemma}\label{Q}\label{tr2}
Let $\FM$ be an indecomposable $p$-permutation $kG$-module with vertex $P$, and let $R\unlhd G$ 
be a $p$-subgroup such that $R\leq \ker(\FM)$. Assume that $Q\leq P$.
Then for any $K$ with $(Q\cap R) \not \leq K < Q$, we have that $\tr_K^{Q}(\FM^K)=0$.
\end{Lemma}

\begin{proof}
Since $(Q\cap R)\not\leq K$, we have $K < (Q\cap R)K\leq Q$. The representatives of the set of left cosets 
of $K$ in $(Q\cap R)K$ can be chosen from $R$ and the number of the left cosets is divisible by $p$. 
So, since $R$ acts trivially on $\FM$, 
$$\tr_K^{(Q\cap R)K}(m)=\sum_{r} r m=\sum_{r} m=0 \text{ for every }m\in \FM^K$$
where $r$ runs over the different left cosets of $K$ in $(Q\cap R)K$. Then 
$$\tr_K^{Q}(\FM^K)= (\tr_{(Q\cap R)K}^{Q}\circ\tr_K^{(Q\cap R)K})(\FM^K) =0$$ as desired.
\end{proof}

\begin{Lemma}\label{correspondence}
Assume that $R\unlhd G$ and $Q\leq G$.
\begin{enumerate}
\item[\rm (i)] For $K$ with $(Q\cap R)\leq K\leq Q$, $KR=QR$ if and only if $K=Q$. 
\item[\rm (ii)] There is a bijective correspondence between the following 
sets: $I_1:=\{K \ | \ (Q\cap R) \leq K <Q\}$ and $I_2:=\{H \ | \ R\leq H <QR\}$ 
given by $K \mapsto KR$.
\end{enumerate}
\end{Lemma}

\begin{proof} 

(i)
We always have that if $K=Q$, then $KR=QR$.
Now suppose that $KR=QR$. It follows that for any $q\in Q$, 
there exist $\kappa\in K$ and $r\in R$ such 
that $q=\kappa r$. Then $r= \kappa^{-1} q$ and $K\leq Q$ implies that 
$r\in Q$. Thus $r\in (Q\cap R)$. Since $(Q\cap R) \leq K$, 
we have that $r\in K$, so that $q=\kappa r\in K$. Hence $Q\leq K$. 

(ii)
First, by (i) the correspondence is well-defined (a map).

Surjectivity:
Take any $H\in I_2$. Set $K:=Q\cap H$. Then, $Q\cap R\leq Q\cap H=:K\leq Q$. Further,
if $K=Q$, then $Q\cap H=:K=Q$ so $Q\leq H$ so that $QR\leq HR=H$ since $R\leq H$, that
contradicts the fact that $H\in I_2$. Hence $K<Q$, that shows that $K\in I_1$.  
Now, we claim that $KR=H$. Obviously $KR\leq HH=H$. 
Suppose that $KR<H$. Then there exists an element $h\in H\backslash KR$. Write
$h=:qr$ for elements $q\in Q$ and $r\in R$ since $H\leq QR$. Hence
$q=hr^{-1}$ and since $qr\,{\not\in}\,( Q\cap H)R$, $q\,{\not\in}\,H$ which means that
$r\,{\not\in}\,H$, a contradiction since $H\in I_2$. Thus, $KR=H$, that yields that
the correspondence is surjective.

Injectivity: Suppose that $KR=K'R$ for $K, K'\in I_1$. Take any $k'\in K'$. Then,
since $k'\in K'R=KR$, there exist 
$\kappa\in K$ and $r\in R$ with $k'=\kappa r$.
Hence, $R\ni r={\kappa}^{-1}k' \in Q$, that yields that ${\kappa}^{-1}k'=r\in Q\cap R\leq K$, 
and hence $k'\in K$. Thus, $K'\leq K$. Similarly we get $K\leq K'$.

\end{proof}

\begin{Lemma}\label{cosets}

Let $R\,\unlhd\, G$ and $Q, K\leq G$ such that $(Q\cap R) \leq K < Q$, and set $n:=|Q:K|$.
Then there exist elements $q_1, \cdots, q_n \in Q$ such that
$Q=\bigsqcup_{i=1}^n\,q_i K$ (disjoint union) and 
$QR=\bigsqcup_{i=1}^n\,q_i KR$ (disjoint union).

\end{Lemma}

\begin{proof}

Since $Q=\bigsqcup_{i=1}^n \,q_i K$, we have that $QR=\bigcup_{i=1}^n \,q_i KR$. 
Suppose there exist $t$ and $s$ such that $q_t KR=q_s KR$. Then it follows that 
$q_s^{-1} q_t\in KR$ and thus $q_s^{-1} q_t\in (Q\cap KR).$ Note that since $K<Q$, 
we have that $Q\cap KR=K(Q\cap R)$, hence $q_s^{-1} q_t\in K (Q\cap R)=K.$ Thus 
$q_t K=q_s K$. It follows that $QR=\bigsqcup_{i=1}^n \,q_i KR$. 
\end{proof}

\begin{Proposition}\label{timesR}
Let $\FM$ be an indecomposable $p$-permutation $kG$-module with vertex $P$, and let $R\unlhd G$ 
be a $p$-subgroup such that $R\leq \ker(\FM)$. Assume that $Q\leq P$.
\begin{enumerate}
\item[\rm (i)]
It holds that $\FM(QR)=\FM(Q)$
as $k$-vector spaces and that $\Res_{N_G(Q)}^{N_G(QR)}\,\FM(QR) \cong \FM(Q).$  
\item[\rm (ii)]

If $\Res^{N_G(QR)}_{C_G(Q)}\, \mathfrak M(QR)$ is indecomposable, then
$\Res^{N_G(Q)}_{C_G(Q)}\, \mathfrak M(Q)$ is indecomposable.
\end{enumerate}
\end{Proposition}
\begin{proof}
(i)
Let $(Q\cap R) \leq K <Q$, then by Lemma \ref{correspondence}(i), we have that 
$KR < QR$. 
Note that $\FM^K=\FM^{KR}$ since $R\leq \ker(\FM)$. 
Let $q_1,\cdots, q_n$ be the same as in Lemma \ref{cosets}.

Then for any $m\in \FM^K$, 
$\tr_K^Q(m)=\sum_{i=1}^n q_i m=\tr_{KR}^{QR}(m).$ Hence, 
\begin{equation}\label{1}\tr_K^Q(\FM^K)=\tr_{KR}^{QR}(\FM^{KR}).\end{equation}

Further by Lemma \ref{tr2}, we have that 
\begin{equation}\label{2}
\sum_{K<Q} \tr_K^{Q}(\FM^K)=\sum_{(Q\cap R)\leq K <Q} \tr_K^{Q}(\FM^K).\end{equation}

It follows from Lemma \ref{tr1} that 
\begin{equation}\label{3} \sum_{H<QR} \tr_H^{QR}(\FM^H)=\sum_{R\leq H < QR}\tr_H^{QR}(\FM^H).\end{equation} 
Now, Lemma \ref{correspondence}(ii) implies that 
\begin{equation}\label{4}
\sum_{R\leq H < QR}\tr_H^{QR}(\FM^H)=\sum_{(Q\cap R)\leq K<Q}\tr_{KR}^{QR}(\FM^{KR}).\end{equation}

Thus,
\begin{align*}
\sum_{K<Q} \tr_K^{Q}(\FM^K)=&\sum_{(Q\cap R)\leq K <Q} \tr_K^{Q}(\FM^K)\text\quad\text{ by (\ref{2})}
\\=&\sum_{(Q\cap R)\leq K <Q}\,\tr_{KR}^{QR}(\FM^{KR})\quad\text{ by (\ref{1})}
\\=&\sum_{R\leq H < QR}\tr_H^{QR}(\FM^H)\quad\text{ by (\ref{4})} 
\\=&\sum_{H<QR} \tr_H^{QR}(\FM^H)\quad\text{ by (\ref{3})}.
\end{align*}
 Therefore, we have that 
$$
\FM(Q)=\FM^Q\big{/} \big{(}\sum_{K < Q} \tr_K^Q(\FM^K)\big{)} 
=\FM^{QR}\big{/} \big{(}\sum_{H<QR} \tr_H^{QR}(\FM^H)\big{)} = \FM(QR)
$$
as sets. But the left and right hand sides, respectively, are a $k N_G(Q)$-module 
and a $k N_G(QR)$-module. Then,
the fact that $N_G(Q)\leq N_G(QR)$ 
implies that the right hand side can also be seen as a $kN_G(Q)$-module, which
competes the proof.

(ii)
Suppose that $\Res^{N_G(Q)}_{C_G(Q)}\,\FM(Q)$ is not indecomposable. 
By (i), we have that $\FM(QR) \cong \FM(Q)$ 
as $k\,N_G(Q)$-modules. So,
$\Res^{N_G(QR)}_{C_G(Q)}\,\FM(QR)$ is not indecomposable.
But since $C_G(QR)\leq C_G(Q)$,
it follows that
$\Res^{N_G(QR)}_{C_G(QR)}\,\FM(QR)$ is not indecomposable,
a contradiction.
\end{proof}

\section{Proof of Main Theorem and Corollary}
Even in this section we set $M:=\Sc(G,P)$ for a $p$-subgroup $P\leq G$.

\begin{proof}[{\bf Proof of Theorem~\ref{MainResult}}]
$(1)\Rightarrow (2):$ Follows from the definition of Brauer indecomposability.
 
$(2)\Rightarrow (1):$ Let $Q$ be an arbitrary subgroup of $P$. Then $R \leq QR\leq P$. 
By the assumption $\Res^{N_G(QR)}_{C_G(QR)}\,M(QR)$ is 
indecomposable. Hence, 
Proposition \ref{timesR}(ii)
implies that $M(Q)$ is indecomposable as a 
$k C_G(Q)$-module. Hence (1) follows.
\end{proof}

\begin{proof}[{\bf Proof of Corollary~\ref{indexP}}]
$(1)\Rightarrow (2):$ Follows from the definition of Brauer indecomposability.
 
$(2)\Rightarrow (1):$
Since $p\,{\not |}\,|N_G(P)/P C_G(P)|$, \cite[Lemma 4.3(i)]{KKM} 
implies that $\Res^{N_G(P)}_{C_G(P)}\,M(P)$ 
is indecomposable. So the assertion follows from the assumption $(2)$ and Theorem \ref{MainResult}.
\end{proof}

\begin{Remark}
Theorem \ref{MainResult} gives us a hidden idea about the reason
why the Scott modules in Ishioka's examples \cite[Examples 3.1 and 3.2]{I} are {\it not} 
Brauer indecomposable. 
Actually we do have the following two examples:
\end{Remark}

\begin{Example}\label{Ishioka2}
This is Example 3.2 in Ishioka's paper \cite{I}. Let
$$D_8:=\langle a, y, z  | \ a^2=y^2=z^2=1, [a, z]=[y,z]=1, [a, y]=z \rangle,$$
$$A_4:=\langle t, b, c  | \ t^3=b^2=c^2=1, [b, c]=1, b^t=c, c^t=bc \rangle \text{ and }$$ 
$$G:=\langle a, y, z \rangle \times \langle t, b, c \rangle \cong D_8 \times A_4.$$
Further set $x:=ab, \ R:=\langle y, z \rangle,\ P:=R \rtimes \langle x \rangle$. Then 
we have that $P\cong D_8$.
Moreover, 
$$G=(\langle b, c \rangle \rtimes \langle t \rangle) \rtimes P$$
where $R$ and $ab$, respectively, act on $\langle b, c \rangle \rtimes \langle t \rangle=A_4$ 
trivially and by conjugation.
From \cite{I}, we have that $\CF_P(G)$ is saturated and $M=\Ind_P^G\,k$. 
Hence, \cite[Lemma 4.3 (ii)]{KKM} implies that $M(P)$ is indecomposable as 
a $kC_G(P)$-module.
Note that $R\unlhd G$. Moreover since $M$ has a basis consisting of the left cosets of $P$ 
in $G$, we have that $R \leq \ker(M)$. 
It holds also that $C_G(R)=R \times A_4$ and so $C_G(R)\,P=G$. Hence 
$$\Res_{C_G(R)}^G M=(\Res_{C_G(R)}^G\circ\Ind_P^G)\,(k)=\Ind_{C_P(R)}^{C_G(R)}\,k
=\Ind_R^{R\times A_4}\,k \cong k A_4$$
and this is not indecomposable. 
 So by Corollary \ref{indexP} or from the definition of Brauer indecomposability, it follows that 
$M$ is not Brauer indecomposable. 

Moreover, by using the results in this paper, we can analyze the other Brauer quotients. For instance,
letting $Q=\langle x, z \rangle$,
since $QR=P$ and since $M(P)$ is $kC_G(P)$-indecomposable, it follows from Proposition  \ref{timesR}(ii)
 that $M(Q)$ is indecomposable as a $kC_G(Q)$-module.

\end{Example}

The following is an easy application of Corollary \ref{indexP}.

\begin{Example} Set $p:=2$,
$$D:=\langle a, y, z  | \ a^2=y^2=z^2=1, [a, z]=[y,z]=1, [a, y]=z \rangle \cong D_8\text{ and}$$ 
$$A_4:=\langle t, b, c  | \ t^3=b^2=c^2=1, [b, c]=1, b^t=c, c^t=bc \rangle.$$ 
Let $G:=D \times A_4$ as the same as previous example, and take 
$$P:=D \times \langle b \rangle \cong D_8 \times C_2.$$
Then $N_G(P)/P\,C_G(P)= (D \times (C_2\times C_2))/(D\times (C_2\times C_2))\cong 1$. 
We claim that $M=\Ind_P^G\,k$. 
Note that for 
$N:=N_G(P)=D \times V_4 \,\unlhd\, G$, the $kN$-module $V:=\Ind_P^N\,k$ is indecomposable 
by Green's 
Theorem. 

We know easily by \cite[II Theorem~12.4 and II Lemma~12.6]{Lan} that
$$\Ind_{C_2}^{A_4}\,k = \boxed
{\begin{matrix}
\ \ k \ \ 1_\omega \ \ 1_{\omega^2} \\
1_\omega \ 1_{\omega^2} \ k
\end{matrix}
}
\ \ \text{ (Loewy and socle series)}
$$
where 
$1_\omega$ and $1_{\omega^2}$ respectively are
non-trivial one-dimensional $kA_4$-modules corresponding to a primitive $3$-rd root of unity
and its dual
(actually this corrects a tiny mistake in \cite[(81.15)Lemma(iv)]{CR}, and we thank
 B. K{\"u}lshammer, who informed the mistake to us).
This implies that $\Sc(A_4, \langle b\rangle) = \Ind_{C_2}^{A_4}\,k$. Hence
$$
\Ind_N^G\,V=\Ind_P^G\,k=(\Inf_{G/D}^G\circ\Ind_{P/D}^{G/D})(k)
 = \Inf_{G/D}^G\Big(\Sc (A_4,C_2)\Big),
 $$ 
and hence this is indecomposable.
Thus
$M=\Ind_P^G\,k$.
Let $R:=D \cong D_8$
Then $R\unlhd G$ and $|P:R|=2$. Also we have $R\leq \ker(M)$ since $R$ acts trivially 
on the set of left cosets of $P$ in $G$ which is a basis for $M$. 
Now $C_G(R)= \langle z \rangle \times A_4$ and $G=P\,C_G(R)$
and $C_P(R)=\langle z \rangle \times C_2$, so that 
$$\Res_{C_G(R)}^G\,M=\Ind_{C_P(R)}^{C_G(R)}\,k
=\Ind_{ \langle z \rangle \times C_2}^{ \langle z \rangle \times A_4}\,k=
\Sc( \langle z \rangle \times A_4,  \langle z \rangle \times C_2),$$
which is an indecomposable $k C_G(R)$-module. Moreover, $N_G(P)/P\,C_G(P)$ 
is a $2'$-group. Therefore Corollary \ref{indexP}
tells us that $M$ is Brauer indecomposable.
\end{Example}

{\small
}

\end{document}